\input epsf

\documentclass[12pt]{article}
\usepackage{floatrow}
\newfloatcommand{capbtabbox}{table}[][\FBwidth]

\usepackage{bm}
\usepackage{mathrsfs}
\usepackage{subfig, caption, keyval}
\usepackage{multirow}
\usepackage{color}
\usepackage{latexsym,amsmath,amsfonts,amscd,amsthm}
\usepackage{epsfig}
\usepackage{changebar}
\numberwithin{equation}{section}

\theoremstyle{plain}
\newtheorem{thm}{Theorem}[section]
\newtheorem{lem}[thm]{Lemma}

\theoremstyle{definition}

\newcommand{\bV}{{\bf V}}
\newcommand{\bn}{{\bf n}}
\newcommand{\bx}{\mathbf{x}}

\begin{document}
\baselineskip=1.5pc

\vspace{.5in}

\begin{center}

{\large\bf Krylov implicit integration factor discontinuous Galerkin methods on sparse grids for high dimensional reaction-diffusion equations}

\end{center}

\vspace{.1in}

\centerline{
Yuan Liu  \footnote{Department of Mathematics and Statistics, Mississippi State University, Mississippi State,
MS 39762, USA.
E-mail: yliu@math.msstate.edu. Research supported by a grant from the Simons Foundation (426993, Yuan Liu).} \qquad
Yingda Cheng  \footnote{Department of Mathematics, Department of  Computational Mathematics, Science and Engineering, Michigan State University,
               East Lansing, MI 48824, USA.
E-mail: ycheng@msu.edu. Research is supported by NSF grants  DMS-1453661, DMS-1720023 and the Simons Foundation under award number 558704.} \qquad
Shanqin Chen \footnote{Department of Mathematical Sciences, Indiana University South Bend, South Bend, IN 46615, USA.
E-mail: chen39@iusb.edu} \qquad
Yong-Tao Zhang \footnote{Department of Applied and Computational Mathematics and Statistics,
University of Notre Dame, Notre Dame, IN 46556, USA. E-mail: yzhang10@nd.edu. Research supported by NSF grant DMS-1620108.
}
}

\vspace{.4in}

\centerline{\bf Abstract}

Computational costs of numerically solving multidimensional partial differential equations (PDEs)
increase significantly when the spatial dimensions of the PDEs are high, due to large number of spatial grid points.
For multidimensional reaction-diffusion equations, stiffness of the
system provides additional challenges for achieving efficient numerical simulations.
In this paper, we propose a class of Krylov implicit integration factor (IIF) discontinuous Galerkin (DG) methods on sparse grids to solve reaction-diffusion equations on high spatial dimensions.
The key ingredient of spatial DG discretization is the multiwavelet bases on nested sparse grids, which can significantly reduce the numbers of degrees of freedom. To deal with the stiffness of the DG spatial operator in discretizing reaction-diffusion equations, we apply the efficient IIF time discretization methods, which are a class of exponential integrators. Krylov subspace approximations are used to evaluate the large size matrix exponentials resulting from IIF schemes for solving PDEs on high spatial dimensions. Stability and error analysis for the semi-discrete scheme are performed. Numerical examples of both scalar equations and systems
in two and three spatial dimensions are provided to demonstrate the
accuracy and efficiency of the methods. The stiffness of the reaction-diffusion equations is resolved well and
large time step size computations are obtained.

\bigskip

\bigskip



\pagenumbering{arabic}

\section{Introduction}
\label{sec1}
\setcounter{equation}{0}
\setcounter{figure}{0}
\setcounter{table}{0}

One class of widely used mathematical models for many physical and biological applications are the reaction-diffusion systems which take the following general form
\begin{equation} \label{eq1}
\frac{\partial u(\mathbf{x}, t)}{\partial t} = \nabla \cdot (\mathbf{k}(\mathbf{x}) \nabla u)  + f(u(\mathbf{x},t)), \qquad \mathbf{x} \in \Omega,
\end{equation}
subject to initial and periodic or Dirichlet boundary conditions. $\Omega=[0,1]^d$ is the spatial domain which is a $d$-dimensional hypercube. $u$ usually represents concentrations of a group of physical or biological species, $\mathbf{k}(\mathbf{x}) = (k_{i, j}(\mathbf{x}))_{1 \le i, j \le d}$ represents the diffusion rates of the species, and $f(u(\mathbf{x},t))$ describes the chemical or biological reactions.  We assume that $\mathbf{k}(\mathbf{x})$ is symmetric positive definite and bounded below and above uniformly, i.e. there exists positive constant $K_0$, $K_1$ such that for any $\mathbf{x} \in \Omega$,
\begin{equation}\label{eq1_1}
K_0 \mathbf{x} \cdot \mathbf{x} \le \mathbf{k(x)}\mathbf{x}\cdot \mathbf{x} \le K_1 \mathbf{x} \cdot \mathbf{x}.
\end{equation}
 Typical applications of the system include the spatial pattern formation such as Gray-Scott model \cite{gray1983autocatalytic, mazin1996pattern}, Schnakenberg model \cite{hundsdorfer2013numerical, iron2004stability}, Fitzhugh Nagumo model \cite{izhikevich2006fitzhugh, jones1984stability}, and many others (e.g., \cite{aragon2002turing, gierer1972theory, ZLN, HHHLSZ}).

Many state-of-the-art numerical methods have been developed to solve the reaction-diffusion systems (\ref{eq1}). Among them, exponential integrators are an efficient class of time-stepping methods to resolve stiffness in the reaction-diffusion systems. The linear diffusions are integrated exactly in exponential integrators, so the stability
 constraint on the time step size caused by linear diffusion terms is totally removed.
 To resolve the constraint on the time step size caused by stiff nonlinear reaction terms,
 implicit integration factor (IIF) methods were developed in \cite{NZZ,NWZL}.
 IIF methods are a class of exponential integrators.
 A novel property of the methods is that the implicit treatment of stiff nonlinear reaction terms is free of the matrix exponential operation, hence the exact integration of the diffusion terms and the implicit treatment of stiff reactions are decoupled. This approach significantly simplifies the resulting nonlinear algebraic system and leads to efficient computations. It also distinguishes IIF methods from standard fully implicit schemes and other implicit exponential integrators (e.g., implicit exponential time differencing (ETD) methods \cite{B2}).

The major difficulty in applying the IIF methods to high spatial dimension PDEs is how to efficiently compute the matrix exponential since the differential
matrix from the high dimensional spatial discretization is extremely large. One approach is to use the compact IIF methods \cite{NWZL, LiuNie, Zhao, WangZhangNie, WangChenNie}. Another approach is to use Krylov subspace approximations.
In \cite{chen2011krylov}, Krylov IIF methods were developed to efficiently implement the IIF schemes for solving multidimensional reaction-diffusion PDEs, where the Krylov subspace approximations to the matrix exponential operator were applied to IIF schemes. Furthermore, it was shown that Krylov IIF methods can be designed to efficiently solve stiff advection-diffusion-reaction systems \cite{JZ1,JZ2}, and they are
especially efficient to solve diffusion problems with cross derivative terms \cite{LZ2}. In \cite{LZ1}, Krylov IIF schemes
were coupled with finite difference methods on sparse grids to solve high dimensional convection-diffusion equations  using the combination technique,
in which it was shown that the efficiency of Krylov IIF schemes for solving high dimensional problems can be improved significantly in high dimensions (see e.g. \cite{smolyak1963quadrature, BG}). For higher order PDEs such as fourth-order equations, the Krylov IIF schemes were shown to be also effective to resolve their stiffness \cite{MZ}.

For the spatial discretizations,
discontinuous Galerkin (DG) finite element methods are a class of popular high order accuracy
methods for numerically solving various PDEs. DG methods have common advantages shared by all finite element methods such as their ability for easy
handling of complicated geometry and boundary conditions. At the same time, they have lots of
flexibility which is not possessed by continuous Galerkin finite element methods. For example, DG methods are
 flexible for easy $hp$ adaptivity which includes changes of approximation orders between neighboring elements and
allowing general meshes with hanging nodes. DG schemes are compact hence easy to efficiently implement on parallel computers. In \cite{chen2011krylov}, Krylov IIF DG methods were designed to solve stiff reaction-diffusion equations.
It was shown that the Krylov IIF methods can resolve very well the stiffness of the DG spatial discretization for reaction-diffusion PDEs which have higher than first order spatial derivatives.
As that for other PDEs with high spatial dimensions,
 a major challenge for numerically solving high dimensional reaction-diffusion equations is the demanding requirement due to the curse of dimensionality. Generally, the computational cost and storage are as large as $O(h^{-d})$ when solving a $d$-dimensional problem, where $h$ represents the mesh size in one spatial direction. To break the curse of dimensionality, there have been attempts using sparse grid  finite element method \cite{PTS}, adaptive wavelet method \cite{abgv} to compute parabolic problems and Black-Scholes equation in option pricing \cite{BHPS}.
On the other hand, the sparse grid DG methods  were developed in \cite{wang2016sparse, guo2016sparse} to efficiently solve high dimensional elliptic equations and kinetic equations. The key ingredient is the multiwavelet bases on hierarchical grids, which can significantly reduce the numbers of degrees of freedom in the spatial DG discretization.
In this paper, we apply the multiwavelet bases on nested sparse grids to
Krylov IIF DG methods, and design the sparse grid Krylov IIF DG methods to solve high dimensional reaction-diffusion equations efficiently. The diffusion operator is discretized by the sparse grid interior penalty (IP) DG method introduced in  \cite{wang2016sparse} for elliptic problems. The stiffness of high dimensional reaction-diffusion equations is resolved well and
large time step size computations are obtained.

The rest of the paper is organized as following. In Section 2, we describe the DG spatial discretization on sparse grids for reaction-diffusion PDEs and perform stability and error analysis for the semi-discrete scheme. Krylov IIF schemes for solving the resulting stiff systems are presented in Section 3. Numerical experiments including both linear and nonlinear examples are reported in
Section 4. Discussions and conclusions are given in Section 5.

\section{Sparse grid DG method in space}
In this section, we will define the semi-discrete sparse grid IPDG method, and provide stability and error analysis. The main ingredients of the scheme include the sparse grid DG finite element space \cite{wang2016sparse}, and the classical IPDG formulation for parabolic problems \cite{arnold}.

\subsection{Formulation of the scheme}

 First, we review the sparse grid DG finite element space introduced in \cite{wang2016sparse}. 
 The construction starts from one dimension. On the interval $[0, 1]$, we define the ``nested partition",
 \begin{equation}\label{sp4}
 I^j_n = (2^{-n}j, 2^{-n}(j+1)], \qquad j=0, 1, ... , 2^n-1,
 \end{equation}
 which consists of uniform cell of size $2^{-n},$ where $n$ denotes the level of the grid. On each level $n$, let
  \begin{equation}\label{sp5}
  V^k_n = \{v | v\in P^k(I^j_n), \quad j=0, 1, 2, ..., 2^n-1 \}
  \end{equation}
  be the usual piecewise polynomial space of degree less than or equal to $k.$ It is obvious that
  \begin{equation}\label{sp6}
  V^k_n \subset V^k_{n+1}, \qquad n=0, 1, 2, ....
  \end{equation}
  We then define the multiwavelet subspace $W^k_n$ as the orthogonal complement of $V^k_{n-1}$ in $V^k_n$ in terms of $L^2$ inner product, i.e.
  \begin{equation}\label{sp7}
  V^k_{n-1} \oplus W^k_n = V^k_n, \qquad  W^k_n \perp V^k_{n-1}.
    \end{equation}
  Let $W^k_0 = V^k_0$ and we will have
  \begin{equation}\label{sp8}
  V^k_N = \oplus_{0\le n \le N} W^k_n.
  \end{equation}
 When $n=0$, the scaled Legendre polynomials are chosen as the basis functions. 
For $n \ge 1$, we define the bases on $W^k_n$ by dilation and translation of the orthonormal   multiwavelet bases $f_p(\cdot)$ given in \cite{alpert1993class}, i.e.,
\begin{equation}\label{sp9}
 v^j_{p, n} = 2^{\frac{n}{2}} f_p(2^l x-(2j+1)), \quad p=1, 2, ..., k+1,  \quad j=0, 1, ... , 2^{n-1}-1.
\end{equation}
It is easy to check that the bases are orthonormal.

 We are now ready to define the sparse grid DG finite element space by tensor product in $d$-dimensional space with $d\ge2.$
Assume that $W^k_{l_m, x_m}$ is the multiwavelet spaces in $x_m$-direction, we let
 \begin{equation}\label{sp10}
 \textbf{W}^k_\textbf{l} = W^k_{l_1, x_1} \times W^k_{l_2, x_2} \times ... \times W^k_{l_d, x_d},
 \end{equation}
 where $\bm{l} = (l_1, ... , l_d)$ is a multi-index, and  $|\bm{l}|_1 = \sum^d_{i=1} l_i.$
 Correspondingly, the basis functions for $\textbf{W}^k_\textbf{l}$ are defined by a tensor product construction
\begin{equation}\label{sp14}
v^{\mathbf{j}}_{\mathbf{p}, \mathbf{l}} (\mathbf{x}) = \prod^d_{m=1} v^{j_m}_{p_m, l_m} (x_m), \quad p_m = 1, 2, ..., k+1, \quad j_m = 0, ..., \max(0, 2^{l_m-1}-1).
\end{equation}
The sparse grid DG finite element approximation space is defined as
 \begin{equation}\label{sp13}
 \hat{\textbf{V}}^k_N = \oplus_{|\textbf{l}|_1 \le N} \textbf{W}^k_\textbf{l}.
\end{equation}
This is a piecewise polynomial space that is defined on $\Omega_N,$ which is a uniform partition of $\Omega$ with $2^N$ cells in each dimension.
In  \cite{wang2016sparse}, it was shown that the number of degrees of freedom of $\hat{\textbf{V}}^k_N$ scales as $O((k+1)^d2^NN^{d-1})$ or $O((k+1)^d h_N^{-1} (\log h_N)^{d-1})$, where $h_N=2^{-N}$ denotes the finest mesh size in each direction. Hence, numerical simulation  employing this space is much more efficient compared with the one using traditional full grid space. The approximation results of this space were established in \cite{wang2016sparse, guo2016sparse}, which shows that a smooth solution can be   approximated well.


\medskip

Now we are ready to define the semi-discrete scheme for the general reaction diffusion equations (\ref{eq1}). We use the classical IPDG formulation \cite{arnold}, i.e. we look for $u_h \in \hat{\textbf{V}}^k_N,  k\ge1$ such that for any test function $v \in \hat{\textbf{V}}^k_N$,
\begin{align}\label{sp2}
\int_{\Omega} (u_h)_t v \ d\textbf{x}  +  B(u_h, v)  = L(v),
\end{align}
where the bilinear form is defined as
$$
B(u_h, v)= \int_{\Omega} \mathbf{k} \nabla u_h \cdot \nabla v \ d\textbf{x}  - \sum_{e\in \Gamma } \int_{e} \{\mathbf{k} \nabla u_h \} \cdot [v] \ ds -\sum_{e \in \Gamma} \{\mathbf{k}\nabla v \} \cdot [u_h] \ ds+  \sum_{e \in \Gamma} \frac{\sigma}{h} \int_e [u_h] \cdot [v] \ ds,
$$
and
$$
L(v)=\int_{\Omega} fv \ d \textbf{x}
$$ for periodic boundary condition, and
$$
L(v)= \int_\Omega f v d\textbf{x}-\int_{\partial \Omega} \left (\mathbf{k} \nabla v \cdot \textbf{n}+\frac{\sigma}{h} v \right ) g \, ds,
$$ for Dirichlet boundary condition $u(\mathbf{x})=g(\mathbf{x}).$

Here, $\Gamma $ is the union of the boundaries for all the fundamental cells  in the partition and $\sigma$ is the penalty parameter, which is taken to be $20$ in this paper. The average and jump are defined as,
\begin{align}\label{sp3}
[q]=q^-n^-+q^+n^+, \qquad & \{ q\} = \frac{1}{2} (q^-+q^+),  \nonumber \\
[\textbf{q}] = \textbf{q}^- \cdot \textbf{n}^-  + \textbf{q}^+ \cdot \textbf{n}^+ ,  \qquad & \{ \textbf{q}\} = \frac{1}{2}(\textbf{q}^- + \textbf{q}^+).
\end{align}
where $\textbf{n}$ is the unit normal. `-' and `+' represent that the directions of the vector point to interior and exterior at the element $e$ respectively.  The definition \eqref{sp3} extends naturally to the boundary for periodic boundary conditions. For Dirichlet boundaies, if $e$ is part of the boundary, then we let $[q] = q \textbf{n}$ ($\textbf{n}$ is the outward unit normal) and $\{\textbf{q} \} = \textbf{q}$. 

\subsection{Stability and error analysis }
In this subsection, we will briefly discuss about the properties of the semi-discrete scheme \eqref{sp2}. The analysis follows closely to   \cite{arnold, wang2016sparse}. We use $\|\cdot\|$ to denote the standard $L^2$ norm on $\Omega$, and
  define the  energy norm of a function $v \in H^2(\Omega_N)$ by
$$ ||| v ||| ^2 :=\sum_{\substack{T \in \Omega_N}}  \int_{T} |\nabla v|^2 \,d\bx \, + \sum_{\substack{e \in \Gamma}}h_N \int_{e} \left \{  \frac{\partial v}{\partial \bn} \right \}^2\,ds\, + \sum_{\substack{e \in \Gamma}}\frac{1}{h_N} \int_{e} [v]^2\,ds.$$

We review some basic properties of the bilinear operator $B(\cdot, \cdot).$
\begin{lem}[Boundedness \cite{arnold}]
\label{lem:bound}
There exists a positive constant $C_b$, depending only on $K_1, \sigma$, such that
$$|B( w, v )| \le C_b |||w |||\cdot ||| v |||,   \quad \forall \,w, v \in H^2(\Omega_N).$$
\end{lem}

\begin{lem}[Coercivity \cite{arnold}]
\label{lem:stab}
When $\sigma$ is taken large enough, there exists a positive constant $C_s$ depending only on $K_0$, such that
$$B( v, v ) \ge C_s  ||| v |||^2,   \quad \forall \,  v \in \hat{\bV}_N^k.$$
\end{lem}

The semi-discrete stability results can be obtained in the same way as for standard IPDG method for parabolic problems.
\begin{thm}[Stability]
Let   $u_h$ be the numerical solution to \eqref{sp2}. Then we have that
\begin{equation}
\|u_h(T)\|^2+C_s\int_0^T |||u_h|||^2 dt \le \|u_h(0)\|^2 +C \int_0^T \|f\|^2 dt
\end{equation}
for periodic boundary condition, and
\begin{equation}
\label{stable}
\|u_h(T)\|^2+C_s\int_0^T |||u_h|||^2 dt \le \|u_h(0)\|^2 +C \int_0^T \|f\|^2 dt+\frac{C}{h_N} \int_0^T\|g\|^2_{L^2(\partial \Omega)} dt
\end{equation}
for Dirichlet boundary condition, where $C$ is a generic constant independent of $N$ and $u$.
\end{thm}
\begin{proof} Without loss of generality, we only prove for the Dirichlet boundary case.
Let $v=u_h$ in \eqref{sp2}, we get $\frac{1}{2}\frac{d}{dt}\|u_h\|^2+B(u_h, u_h)=L(u_h).$ Using Lemma \ref{lem:stab},  Cauchy-Schwarz and Young's inequality, we have
$$
\frac{1}{2}\frac{d}{dt}\|u_h\|^2+C_s  |||u_h|||^2 \le L(u_h) \le C \|f\|^2+ \frac{C_s}{2}|||u_h|||+\frac{C}{h_N}\|g\|^2_{L^2(\partial \Omega)}.
$$
Multiplying both sides by $2,$ and integrating in time from $0$ to $T,$ we obtain \eqref{stable}.
\end{proof}

The next theorem provides error estimates of the semi-discrete scheme. As in \cite{wang2016sparse}, we introduce the following norm for a function. For any set $L=\{i_1, \ldots i_r \} \subset \{1, \ldots d\}$, we define $L^c$ to be the complement set of $L$ in $\{1, \ldots d\}.$ For a non-negative integer $\alpha$ and set $L$,  we define the semi-norm
\begin{flalign*}
|v|_{H^{\alpha,L}(\Omega)} :=  \left \| \left ( \frac{\partial^{\alpha}}{\partial x_{i_1}^{\alpha}} \cdots \frac{\partial^{\alpha}}{\partial x_{i_r}^{\alpha}}  \right ) v \right \|_{L^2(\Omega)},
\end{flalign*}
and
$$
|v|_{\mathcal{H}^{q+1}(\Omega)} :=\max_{1 \leq r \leq d} \left ( \max_{\substack{L\subset\{1,2,\cdots,d\} \\|L|=r}} |v|_{H^{q+1, L}(\Omega)} \right ),$$
which is the norm for the mixed derivative of $v$ of at most degree $q+1$ in each direction.
\begin{thm}[Error estimate]
\label{thm:error}
Let $u$ be the exact solution to \eqref{eq1}, and $u_h$ be the numerical solution to \eqref{sp2} with numerical initial condtion $u_h(0)=Pu,$ where $P$ denotes the $L^2$ projection of a function onto the space $\hat{\textbf{V}}^k_N.$ For  $k \geq 1$,   any $1 \leq q \leq \min \{p, k\}$, and any
$u \in L^2(0,T;\mathcal{H}^{p+1}(\Omega))$, $N\geq 1$, $d \geq 2$,   we have
$$
\|u-u_h\| \le C N^{-d}  2^{-Nq}  |u|_{L^2(0,T;\mathcal{H}^{q+1}(\Omega))},
$$
and
$$
(\int_0^T |||u-u_h|||^2 dt)^{1/2}  \le C N^{-d}  2^{-Nq}  |u|_{L^2(0,T;\mathcal{H}^{q+1}(\Omega))},
$$
 where $C$ is a generic constant independent of $N$ and $u$.
\end{thm}
\begin{proof}
Let $e=u-u_h=\eta+\xi,$ where $\eta=u-Pu, \xi=Pu-u_h.$  Using Lemma 3.4 in \cite{guo2016sparse} and trace inequality, we obtain $|||\eta||| \le C N^{-d}  2^{-Nq} |u|_{\mathcal{H}^{q+1}(\Omega)}.$

It is obvious to see that the exact solution also satisfies \eqref{sp2}.
 Therefore, we get the error equation $\int_{\Omega} e_t v \ d\textbf{x}  +  B(e, v) =0.$ Let $v=\xi,$ and use the property of $L^2$ projection, we have
 $$
 \int_{\Omega} \xi_t \xi \ d\textbf{x}  +  B(\xi, \xi) =-B(\eta, \xi).
 $$
 By Lemmas \ref{lem:bound} and \ref{lem:stab}, we obtain
 $$
\frac{1}{2}\frac{d}{dt}\|\xi\|^2+C_s  |||\xi|||^2 \le C_b |||\eta|||\cdot|||\xi|||\le  \frac{C_s}{2}|||\xi||^2+C|||\eta|||^2.
 $$
 Hence,
\begin{eqnarray*}
\|\xi(T)\|^2+C_s  \int_0^T |||\xi|||^2 dt &\le&  \|\xi(0)\|^2+C\int_0^T |||\eta|||^2 dt=C\int_0^T |||\eta|||^2 dt\\
&\le& C N^{-2d}  2^{-2Nq}  |u|^2_{L^2(0,T;\mathcal{H}^{q+1}(\Omega))},
\end{eqnarray*}
and we are done by combining with the estimates of $\xi$ with $\eta$.
\end{proof}

  The results obtained above for the $L^2$ error is sub-optimal. When the exact solution is smooth, the convergence rates in $L^2$ and the energy norm are of order $O((\log h_N)^d h_N^k).$
The standard technique to improve the convergence rate in $L^2$ norm for IPDG method is to use elliptic projection. However, as shown in \cite{wang2016sparse}, we cannot obtain optimal $L^2$ convergence rate for the associated elliptic solver on sparse grid. Therefore, in Theorem \ref{thm:error}, we use the simple $L^2$ projection instead. Numerical examples in Section \ref{numer} indicate an optimal convergence rate up to logarithmic factors, similar as those observed in \cite{wang2016sparse}.

\section{Krylov IIF temporal discretization}

In this section, we present the Krylov IIF methods \cite{chen2011krylov} which are employed to evolve the stiff semi-discretized ODE system resulted from the sparse grid IPDG spatial discretization of the reaction-diffusion PDEs.
The semi-discretized ODE system has the following form
\begin{equation}\label{krylov1}
\frac{d U(t)}{dt} = \mathbf{A}U +F(U),
\end{equation}
where $U(t)$ is the unknown vector, $\mathbf{A}$ is the coefficient matrix derived from the sparse grid DG discretization of the diffusion operator, and $F(U)$ is the nonlinear reaction term. We multiply the integration factor $e^{-At}$ to (\ref{krylov1}) and perform integration on the time interval $[t^n, t^{n+1}]$ with time step size $\Delta t$ to obtain
\begin{equation}\label{krylov2}
U(t^{n+1}) = e^{\mathbf{A}\Delta t}U(t_n) + e^{\mathbf{A}\Delta t} \int^{\Delta t}_0 e^{-\mathbf{A}\tau} F(U(t_n+\tau)) d\tau.
\end{equation}
The integral in (\ref{krylov2}) is approximated by an $(r-1)$-th order Lagrange interpolation polynomial based on numerical values at time level $\{t^{n+1}, t^n, ... t^{n+2-r} \}$, and we obtain the $r$-th order IIF scheme
\begin{equation}\label{krylov3}
U^{n+1} = e^{\mathbf{A}\Delta t}U^n +\Delta t (\alpha_1F(U^{n+1}) + \sum_{i=0}^{r-2} \alpha_{-i}e^{(i+1)\mathbf{A}\Delta t}F(U^{n-i})),
\end{equation}
where
\begin{equation}\label{krylov4}
\alpha_{-i} = \frac{1}{\Delta t} \int^{\Delta t}_0 \prod^{r-2}_{k=-1, k \neq i} \frac{\tau +k\Delta t}{(k-i) \Delta t} d\tau, \qquad -1 \leq i \leq r-2,
\end{equation}
are the Lagrange interpolation coefficients. Among the most often used IIF schemes, the second order IIF scheme (IIF2) is of the following form
\begin{equation}\label{krylov5}
U^{n+1} = e^{\mathbf{A}\Delta t} (U^n +\frac{\Delta t}{2} F(U^n) )+\frac{\Delta t}{2} F(U^{n+1}),
\end{equation}
and the third order scheme (IIF3) is
\begin{equation}\label{krylov6}
U^{n+1} = e^{\mathbf{A}\Delta t} U^n + \Delta t (\frac{5}{12}F(U^{n+1})+\frac{2}{3}e^{\mathbf{A}\Delta t}F(U^n)-\frac{1}{12}e^{2\mathbf{A}\Delta t} F(U^{n-1})).
\end{equation}
In order to compute the matrix exponentials efficiently in IIF schemes for solving multidimensional PDEs,
Krylov IIF schemes \cite{chen2011krylov} use Krylov subspace approximation \cite{ gallopoulos1992efficient} to
evaluate the product of a matrix exponential and a vector such as $e^{\mathbf{A}\Delta t}v$, etc. Let $K_M$ be the dimension $M$ Krylov subspace associated with the matrix $\mathbf{A}$, which is
\begin{equation}\label{krylov7}
K_M = \text{span}\{v, \mathbf{A}v, \mathbf{A}^2v, ..., \mathbf{A}^{M-1}v \}.
\end{equation}
An orthonormal basis $V_M = [v_1, v_2, v_3, ..., v_M]$ for $K_M$ and an $M \times M$ upper Hessenberg matrix $\mathbf{H_M}$ are generated by the Arnoldi algorithm \cite{trefethen1997numerical}. Particularly, $\mathbf{H_M}$ is the projection of the matrix $\mathbf{A}$ onto $K_M$ with respect to the basis $V_M$. Furthermore, since the columns of $V_M$ are orthonormal, we have the approximation
\begin{equation}\label{krylov8}
e^{\mathbf{A}\Delta t}v  \approx \gamma V_Me^{\mathbf{H_M}\Delta t}e_1,
\end{equation}
where $\gamma = ||v||_2$ and $e_1$ is the unit column $M\times 1$ vector with first entry 1. By doing so, the matrix exponential $e^{\mathbf{A}\Delta t}$ problem is replaced by the $e^{\mathbf{H_M} \Delta t}$.
Note that the dimension $M$  of the Krylov subspace is {\bf much} smaller than the dimension of the large sparse
matrix ${\bf A}$. In this paper, the dimension of matrix $A$ is exactly the number of degree of freedom of the DG scheme and is quite large. In our numerical simulations, $M$ is taken to be $25$ or $100$, and accurate results
are obtained as shown in Section 4. Matrix exponentials of small $M \times M$ matrix $\mathbf{H_M}$ are efficiently computed by a scaling and squaring algorithm with a Pad$\acute{e}$ approximation \cite{Higham}.
The computational cost is significantly saved.

\section{Numerical examples}
\label{numer}
In this section, we apply the proposed sparse grid Krylov IIF DG methods to various linear and nonlinear problems.
Both $P^1$ and $P^2$ DG cases are tested.  The Krylov IIF2 scheme is coupled with the $P^1$ DG discretization, and
the Krylov IIF3 scheme is with the $P^2$ one.

The numerical implementation of sparse grid DG methods is quite different from traditional DG methods, since the basis functions in space $\hat{\mathbf{V}}_N^k$ are globally and hierarchically defined. In order to save computational cost, the \emph{unidirectional principle} is employed to efficiently evaluate the multidimensional integrations, i.e. we decouple the multidimensional integrations  into the multiplication of one-dimensional integrals by utilizing the hierarchical structure and orthonormal property of the basis functions. We also note that this procedure takes place one time before the time evolution, which further accelerates the computation. The most costly part in the implementation comes from the  evaluation of nonlinear reaction terms.

\bigskip

\noindent {\bf Example 1.} In this example, we perform comparison study between our proposed (1) Krylov IIF IPDG method on sparse grids, (2) Krylov IIF IPDG method on regular full grids, and (3) Runge-Kutta (RK) IPDG method on sparse grids. We consider a linear
$d$-dimensional diffusion PDE
\begin{align}
\begin{cases}
u_t (\mathbf{x}, t) = k \triangle_{\mathbf{x}} u,    \\
u(\mathbf{x},0)= \prod\limits_{i=1}^{d} \sin(2\pi x_i).
\end{cases}
\end{align}
The PDE is defined on domain $[0, 1]^d$ with periodic boundary condition and $k = \frac{1}{4d\pi^2 }$. In this case, the exact solution is $u(\mathbf{x},t)= e^{-t}\prod\limits_{i=1}^{d} \sin(2\pi x_i)$. We perform computations for the 2D and 3D cases.

%
%
%
%
%
%
The stiffness of the matrix $A$ resulted from the DG spatial discretizations of the diffusion operator is studied.
In the numerical implementation, we use uniform mesh partition in each direction and denote the mesh size as $h_N = 2^{-N}$ where $N$ represents the number of mesh level, then the CFL number for an explicit time integrator is defined as
\begin{equation}
CFL=kd \frac{ \Delta t}{h_N^2}.
\end{equation}
Explicit second order RK (RK2) and third order RK (RK3) methods are first used to show the severe stability constraint due to the stiffness of the diffusion operator.
The numerical CFL numbers to ensure stable computations are reported in Table \ref{table_CFL} for the 2D and 3D cases.
The numerical CFL numbers we obtained for IPDG methods on regular full grids are the same when the mesh is refined, which is consistent with the fact for regular DG methods. While for the sparse grid IPDG methods, the CFL numbers vary with different mesh levels. Comparing with full grid cases, the CFL constraints are less restrictive. Similar observations are made in \cite{tao} for sparse grid DG method for convection operators.   However, the CFL numbers are still very small and tiny time step sizes are needed to achieve stable computations.

To complete the study, in Table \ref{eig1}, we list the largest negative eigenvalues $\lambda_0(A)$ of matrices $A$, which are the $P^1$ and $P^2$ DG discretization matrices for the diffusion operator on the domain $[0, 1]^2$ or $[0, 1]^3$ with periodic boundary conditions, on successively refined mesh levels.  It is obvious to see that the matrices $A$ have quite large magnitude eigenvalues, for different mesh levels on both sparse grids and regular full grids.
When a regular implicit time integrator, such as the backward Euler method, is employed, we need to solve a linear system with the coefficient matrix $(I -\Delta t A)$.
In Table  \ref{eig1}, the condition numbers of the matrices $(I -\Delta t A)$ are presented for different refined mesh levels, with {$\Delta t = h_N$}.   We can see that the magnitudes of these condition numbers are also quite large, especially for more refined mesh levels. This introduces additional numerical challenges to solve the linear system efficiently. It is also observed in Table \ref{eig1} that the numbers of unknown degree of
freedom in the spatial discretizations are saved a lot when the sparse grids are used rather than regular full grids.
Hence both CPU times and computer storages can be saved significantly to solve high dimensional problems by using
sparse grids. The discussions above justify the adoption of Krylov IIF methods as the time integrator.
Next we show that the Krylov IIF methods can resolve the stiffness of this DG spatial operator very well.


\begin{table} [htb]
\begin{center}
\caption {Numerical CFL numbers for 2D and 3D IPDG methods, with RK2 and RK3. { N is the mesh level and DOF denotes the degrees of freedom of DG approximation. } }
\begin{tabular}{ccccccccc}
  \cline{2-9}
   &  \multicolumn{1}{c}{\multirow{2}{*}{N}}  & \multicolumn{1}{c}{\multirow{2}{*}{DOF}}   & $P^1$  &  & & \multicolumn{1}{c}{\multirow{2}{*}{DOF}} & $P^2$ \\
  \cline{4-5}
  \cline{8-9}
  &  & & RK2 & RK3  &  & &  RK2 & RK3  \\
 \cline{2-9}
 &    Full grid \\
  & --   &  --   & 0.00416  & 0.00523   & & --  & 0.00208  &  0.00262 \\
\cline{2-9}
&    2D sparse grid \\
  & 3   &  80   & 0.00724  & 0.00900  & & 180 & 0.00388  &  0.00487 \\
  & 4   &  192   & 0.00778 & 0.00978  & & 432 &0.00403  &  0.00506 \\
  & 5   &  448   & 0.00806 & 0.01013  & & 1008 &0.00410  &  0.00515 \\
  & 6   &  1024   & 0.00820 & 0.01030  & & 2304 &0.00413  &  0.00519 \\
  & 7   &  2304   & 0.00826 & 0.01039  & & 5184 &0.00415  &  0.00521 \\
  & 8   &  5120   & 0.00830 & 0.01043  & & 11520 &0.00416  &  0.00522 \\
\cline{2-9}
&    3D sparse grid \\
  & 3   &  304   & 0.00953  & 0.01197  & & 1026 & 0.00540 &  0.00678 \\
   & 4   &  832   & 0.01091  & 0.01370  & & 2808 & 0.00583 &0.00733 \\
    & 5   &  2176  & 0.01170  & 0.01469  & & 7344 &  0.00605&  0.00760 \\
     & 6   &  5504   & 0.01210  & 0.01520  & & 18576  & 0.00615  &  0.00773 \\
\cline{2-9}
  \end{tabular}\label{table_CFL}
\end{center}
\end{table}



\begin{table} [htb]
\begin{center}
\caption {Eigenvalues and condition numbers for the discretization matrix A for the diffusion operator.  { N is the mesh level and DOF denotes the degrees of freedom of DG approximation. } }
\begin{tabular}{ccccccccc}
  \cline{2-9}
   &  \multicolumn{1}{c}{\multirow{2}{*}{N}}  & \multicolumn{1}{c}{\multirow{2}{*}{DOF}}   & $P^1$  &  & & \multicolumn{1}{c}{\multirow{2}{*}{DOF}} & $P^2$ \\
  \cline{4-5}
  \cline{8-9}
  &  & & cond(I-$\Delta t$A)  & $\lambda_0(A)$ &  & &  cond(I-$\Delta t$A) & $\lambda_0(A)$  \\
 \cline{2-9}
 &    2D sparse grid \\
  & 3   &  80  & 1.12E+04  & -3.53E+04   & & 180 & 2.97E+04  &  -6.60E+04 \\
& 4 & 192 & 2.45E+04  & -1.31E+05  & &432 & 8.41E+04 & -2.54E+05 \\
 & 5   & 448 &7.39E+04 &  -5.08E+05  & & 1008 & 2.22E+05 & -9.98E+05\\
 & 6  &  1024  & 2.11E+05  & -2.00E+06 & & 2304 &  5.67E+05 & -3.96E+06 \\
  & 7  &  2304  & 5.36E+05  &   -7.92E+06 & & 5184 & 1.24E+06 & -1.58E+07 \\
   & 8  &  5120  & 1.20E+06 & -3.16E+07&   & 11520 &2.80E+06 &  -6.30E+07 \\
\cline{2-9}
 &    2D full grid \\
  & 3   &  256  & 3.03E+04  & -6.14E+04  &  & 576 & 7.01E+04 &-1.23E+05 \\
& 4 & 1024 & 9.21E+04  & -2.46E+05  & & 2304 & 1.78E+05 & -4.91E+05 \\
 & 5   & 4096 & 2.06E+05  &  -9.83E+05 & & 9216 & 4.41E+05 & -1.96E+06 \\
 & 6  &  16384  & 4.45E+05 &  -3.93E+06  &  & 36864 & 1.01E+06 &-7.86E+06 \\
 \cline{2-9}
  &  3D sparse grid \\
 & 3   &  304 & 1.77E+04  & -4.03E+04 & & 1026 &3.99E+04 &  -7.11E+04 \\
& 4 &832  & 3.44E+04 & -1.41E+05 & & 2808 &  1.09E+05 &  -2.63E+05 \\
 & 5   & 2176 &8.15E+04 & -5.25E+05  & &7344 & 3.42E+05 & -1.02E+06  \\
 & 6  &  5504  & 2.19E+05  &  -2.03E+06 &  & 18576 & 1.12E+06 & -3.99E+06\\
 \cline{2-9}
 &   3D  full grid  &  & \\
 & 2   & 512 &1.41E+04 & -2.30E+04 & & 1728 &5.47E+04 & -4.60E+04 \\
 & 3 & 4096 & 9.00E+04 & -9.21E+04 & &  13824 & 2.18E+05 & -1.84E+05 \\
  \cline{2-9}
  \end{tabular}\label{eig1}
\end{center}
\end{table}

It is important to note that the errors generated by the Krylov subspace approximations
 have become much smaller than the truncation errors of the numerical schemes if the dimension of Krylov subspace is large enough, but still much smaller than the dimension of matrix $A$.
 This is confirmed by the following numerical results which show the numerical errors of computations by Krylov IIF schemes with different dimensions of Krylov subspaces. For a fixed spatial mesh which has the refined level $N=7$, we present the numerical errors at $T=0.6$ in Table \ref{krylov} by sparse grid Krylov IIF DG schemes if different dimensions $M$ of the Krylov subspace are used.
 The numerical results show that in both 2D and 3D cases, as long as the dimension of Krylov subspace is large enough, for instance, larger than $M = 25$ in this example, the numerical errors  only have slight differences for different $M$.
 This is because the numerical errors of this example are mainly due to the truncation
errors of the DG spatial discretizations when $M \geq 25$,
 and the numerical errors due to the Krylov subspace approximations are negligible.
 Results in Table \ref{krylov} also show that
 large time step size proportional to the spatial grid size, i.e., $\Delta t = h_N $, can be used for this parabolic PDE to obtain stable computations. Actually since this example only has linear diffusion terms, in the time direction the Krylov IIF methods can integrate
the DG spatial discretization operator almost ``exactly'' up to the numerical errors of the Krylov subspace approximations. This is also confirmed in Table \ref{krylov}.
A very large time step size $\Delta t = 0.6 $ (i.e., one time step is used to reach the final time $T$) is used, and
there are very little differences between the magnitudes of numerical errors from the computations with $\Delta t = h_N$ and $\Delta t = 0.6 $, if $M \geq 25$.

\begin{table} [htb]
\begin{center}
\caption {$L^2$ numerical errors if different dimensions M of the Krylov subspace are used for refinement level $N=7$ and final time $T=0.6$. { DOF denotes the degrees of freedom of DG approximation.} }
\begin{tabular}{ccccccccccc}
  \cline{2-11}
   &    & &     & $ P^1$  &  & &  & $P^2$ \\
  \cline{3-6}
  \cline{8-11}
  &  &DOF    & $\Delta t = 0.6 $  &  & $\Delta t = h_N $  & &   DOF   &  $\Delta t = 0.6 $ &   & $\Delta t = h_N $  \\
 \cline{2-11}
 &    2D  sparse   &  &\\
 \cline{2-11}
&   $M=10$   &  2304  & 6.29E-03   &   & 9.02E-04 &   & 5184  & 3.01E-06  && 3.51E-06 \\
 &   $M=25$   &  2304  & 8.51E-04   &   & 8.39E-04 &   & 5184  & 3.46E-06  && 3.53E-06   \\
  &   $M=100$   &  2304  & 8.38E-04   &   & 8.39E-04 &   & 5184  & 3.51E-06  && 3.57E-06     \\
   &   $M=250$   &  2304  & 8.39E-04   &   & 8.40E-04 &   & 5184  & 3.55E-06  && 3.57E-06   \\
   &   $M=500$   &  2304  & 8.40E-04   &   &  8.40E-04&   & 5184  & 3.57E-06  & &3.57E-06  \\
\cline{2-11}
&    3D  sparse   &  &\\
 \cline{2-11}
&   $M=10$   &  13568  & 1.94E-01   &   &  6.04E-02&   &  45792 & 3.89E-05 & &2.53E-05  \\
&   $M=25$   &  13568  & 1.40E-02  &   & 1.39E-02 &   &  45792 & 2.26E-05  && 2.53E-05 \\
&   $M=100$   &  13568  & 1.16E-02   &   & 1.15E-02 &   &  45792 & 2.52E-05 &&2.60E-05  \\
&   $M=250$   &  13568  & 1.15E-02   &   &   1.15E-02  & &  45792 & 2.53E-05 &&2.61E-05 \\
&   $M=500$   &  13568  & 1.15E-02   &   &    1.15E-02 & &  45792 & 2.60E-05  && 2.61E-05 \\
\cline{2-11}
   \end{tabular}\label{krylov}
\end{center}
\end{table}

We then investigate the performance of the proposed Krylov IIF sparse grid IPDG methods and carry out the convergence study. In Table \ref{more_step}, $L^2$ errors with accuracy order at $T=2$ are presented, with time step size $\Delta t = h_N$ and { $M=25$}. We can clearly observe that the accuracy order is close to order $k+1$ in each case. The CPU time of time evolution is also reported in Table \ref{more_step}.  Large time step size computations are achieved and the stiffness of DG operator is resolved well. We observe that the CPU time approximately linearly depends on the product of the number of degrees of freedom and the total time steps in each case, which implies linear dependence of CPU cost per time step.

\begin{table} [htb]
\begin{center}
\caption {Accuracy test for {\bf Example 1}. Krylov IIF2 and IIF3 with sparse grid IPDG in space. T=2.0, $\Delta t = h_N$. N is the mesh level and DOF denotes the degrees of freedom of DG approximation.}
\begin{tabular}{cccccccccccc}
  \cline{2-12}
   &  \multicolumn{1}{c}{\multirow{2}{*}{N}}  & \multicolumn{1}{c}{\multirow{2}{*}{DOF}}   & $P^1$  &  && &\multicolumn{1}{c}{\multirow{2}{*}{N}}  & \multicolumn{1}{c}{\multirow{2}{*}{DOF}} & $P^2$ \\
  \cline{4-6}
  \cline{10-12}
  &  & & $L^2$ error  & Order & CPU  & &  & &  $L^2$ error & Order &  CPU  \\
 \cline{2-12}
 &    2D  sparse   &  &\\
 & 4 & 192 & 2.60E-02   & --   & 0.03 & & 3& 180 &2.22E-03   &  -- & 0.01 \\
& 5 & 448 & 7.42E-03   & 1.81   & 0.11 & & 4& 432 &2.76E-04   &  3.01 & 0.07 \\
& 6 & 1024 & 1.91E-03   & 1.96   & 0.51 & & 5& 1008 &3.93E-05   &  2.81 & 0.30 \\
& 7 & 2304 & 4.77E-04   & 2.00   & 2.28 & & 6& 2304 &5.94E-06   &  2.73 & 1.44\\
& 8 & 5120 & 1.18E-04   & 2.01   & 11.23  & & 7& 5184 &8.77E-07   &  2.76 & 7.86 \\
& 9 & 11264 & 2.90E-05   &2.03   & 53.93 & & 8& 11520 &1.26E-07   &  2.80 & 39.31 \\
\cline{2-12}
  \cline{2-12}
  &    3D  sparse   &  &\\
 & 6   &  5504  & 2.54E-02  & --   &2.16& & 5 &7344    &  2.40E-04 & -- & 2.94  \\
  & 7   &  13568  & 6.40E-03  & 1.99   &12.48& & 6 &18576    &  3.80E-05 & 2.66 & 15.64  \\
 & 8   &  32768  & 1.62E-03  & 1.99   &89.55& & 7 &45792   &  6.29E-06 & 2.60 & 77.29 \\
  & 9   &  77824  & 3.94E-04  & 2.04   &506.41& & 8 & 110592    &  1.01E-06 & 2.64 & 505.16    \\
  \cline{2-12}
  \cline{2-12}
   \end{tabular}\label{more_step}
\end{center}
\end{table}

\bigskip
\noindent {\bf Example 2.} We consider a $d$-dimensional linear reaction-diffusion equation
\begin{align}
\begin{cases}
u_t = k \triangle_{\mathbf{x}} u +u-e^{-t}\prod\limits_{i=1}^{d} \sin(2\pi x_i),  \\
u(\mathbf{x},0)= \prod\limits_{i=1}^{d} \sin(2\pi x_i),
\end{cases}
\end{align}
on domain $[0, 1]^d$ with periodic boundary conditions. The exact solution is $u(\mathbf{x},t)= e^{-t}\prod\limits_{i=1}^{d} \sin(2\pi x_i)$.
Again, computations are performed for the 2D ($d=2$) and 3D ($d=3$) cases.
$k = \frac{1}{8\pi^2}$ in 2D and $k = \frac{1}{12\pi^2}$ in 3D.  The numerical simulations are performed up to $T=1.0$ for the 2D case and $T=0.4$ for the 3D case, with the Krylov subspace dimension $M=25$ and time step size $\Delta t = h_N$.  $L^2$ errors and the orders of accuracy are listed in Table \ref{2d3dlinear} for the second order Krylov IIF method with $P^1$ sparse grid IPDG and the third order Krylov IIF method with $P^2$ sparse grid IPDG. Similar as Example 1, the desired numerical orders of accuracy for the Krylov IIF DG schemes are attained in this example.

\begin{table} [htb]
\begin{center}
\caption {Accuracy test for {\bf Example 2}. Krylov IIF2 and IIF3 with sparse grid IPDG in space. $\Delta t = h_N$.
N is the mesh level and DOF denotes the degrees of freedom of DG approximation.}
\begin{tabular}{cccccccccc}
  \cline{2-10}
   &  \multicolumn{1}{c}{\multirow{2}{*}{N}}  & \multicolumn{1}{c}{\multirow{2}{*}{DOF}}   & $P^1$  &  & &\multicolumn{1}{c}{\multirow{2}{*}{N}}  & \multicolumn{1}{c}{\multirow{2}{*}{DOF}} & $P^2$ \\
  \cline{4-5}
  \cline{9-10}
    & & & $L^2$ error  & Order   &   & &&   $L^2$ error & Order   \\
 \cline{2-10}
  &  \multicolumn{2}{c}{2D sparse}  & T=1.0\\
  & 3 & 80 & 1.96E-01   & --   & & 3& 180 &6.20E-03    &  --  \\
   & 4 & 192 & 6.86E-02  & 1.51   & & 4& 432 &7.58E-04     &  3.03  \\
   & 5 & 448 & 1.89E-02  & 1.86   & & 5& 1008 &1.07E-04     &  2.83 \\
   & 6 & 1024 & 5.25E-03  & 1.85   & & 6& 2304 &1.60e-05    &  2.74  \\
   & 7 & 2304 & 1.21E-03  & 2.12   & & 7& 5184 &2.37e-06     &  2.75  \\
  \cline{2-10}
   &  \multicolumn{2}{c}{3D sparse}  & T=0.4\\
    & 5 & 2176 & 2.40E-01   & --   & & 3  &  1026  & 1.57e-02 & --  \\
     & 6 & 5504 & 9.55E-02   & 1.33   & & 4    &  2808  & 6.88e-03 &1.19   \\
      & 7 & 13568 &1.63e-02   & 2.55  & & 5   &  7344  & 1.11e-03 & 2.64   \\
    &    8 &   32768 & 3.89e-03   & 2.06  & &6  &  18576  & 1.86e-04 & 2.57  \\
        & 9 &  77824  & 9.27e-04   & 2.07  & &7 &  45792  & 3.09e-05 &2.59  \\
         \cline{2-10}
   \end{tabular}\label{2d3dlinear}
\end{center}
\end{table}

\bigskip
\noindent {\bf Example 3.} We consider a $d$-dimensional nonlinear reaction-diffusion equation
\begin{align}
\begin{cases}
u_t = k \triangle_{\mathbf{x}} u +u^2-e^{-2t}\prod\limits_{i=1}^{d}  \sin^2(2\pi x_i), \\
u(\mathbf{x},0)=\prod\limits_{i=1}^{d} \sin(2\pi x_i).
\end{cases}
\end{align}
The PDE is defined on the domain $[0, 1]^d$ with zero Dirichlet boundary conditions.
The exact solution is $u(\mathbf{x},t)= e^{-t}\prod\limits_{i=1}^{d} \sin(2\pi x_i)$. We perform computations for the 2D and the 3D cases.
$k = \frac{1}{8\pi^2}$ in the 2D case, and $k = \frac{1}{12\pi^2}$ in the 3D case. The numerical simulations are carried up to $T=1.0$ for the 2D case and $T=0.2$ for the 3D case, with the Krylov subspace dimension $M=25$ and the time step size $\Delta t = h_N$. The nonlinear system resulting from the nonlinear reaction term is solved by the Newton's method. Similar as the previous examples, we present $L^2$ errors and numerical orders of accuracy for this example in Table \ref{2d3dnonlinear} for the second order Krylov IIF method with $P^1$ sparse grid IPDG and the third order Krylov IIF method with $P^2$ sparse grid IPDG. Again, the desired numerical orders of accuracy are obtained, and large time step size (proportional to the spatial grid size) computations are achieved for this multidimensional nonlinear parabolic PDE.

\begin{table} [htb]
\begin{center}
\caption {Accuracy test for {\bf Example 3}. Krylov IIF2 and IIF3 schemes with sparse grid IPDG in space. $\Delta t = h_N$. N is the mesh level and DOF denotes the degrees of freedom of DG approximation.}
\begin{tabular}{cccccccccc}
  \cline{2-10}
   &  \multicolumn{1}{c}{\multirow{2}{*}{N}}  & \multicolumn{1}{c}{\multirow{2}{*}{DOF}}   & $P^1$  &  & &\multicolumn{1}{c}{\multirow{2}{*}{N}}  & \multicolumn{1}{c}{\multirow{2}{*}{DOF}} & $P^2$ \\
  \cline{4-5}
  \cline{9-10}
    & & & $L^2$ error  & Order   &   & &&   $L^2$ error & Order   \\
 \cline{2-10}
  &  \multicolumn{2}{c}{2D sparse}  & T=1.0\\
  & 3 & 80 & 1.96e-01   & --   & & 3& 180 &5.96e-03    &  --  \\
  & 4 & 192 & 4.70e-02   & 2.10   & & 4& 432 &7.33e-04    &  3.02  \\
   & 5 & 448 & 1.22e-02    & 1.94   & & 5& 1008 &1.16e-04 &  2.66 \\
   & 6 & 1024 & 3.10e-03    & 1.98   & & 6& 2304 &1.62e-05  &  2.83\\
    & 7 & 2304 &  7.80e-04   & 1.99   & & 7& 5184 &2.43e-06  &  2.74  \\
  \cline{2-10}
   &  \multicolumn{2}{c}{3D sparse}  & T=0.2\\
    & 5 & 2176 &  2.90e-01 &--  & & 4& 2808&8.57e-03    &  --  \\
     & 6 & 5504 & 9.72e-02  & 1.58    & & 5& 7344&1.36e-03 & 2.65  \\
      & 7 & 13568 & 1.79e-02  & 2.44   & & 6& 18576&2.29e-04 & 2.57  \\
       & 8 & 32768 & 4.13e-03  & 2.12   & & 7& 45792  & 3.82e-05 & 2.58  \\
    \cline{2-10}
   \end{tabular}\label{2d3dnonlinear}
\end{center}
\end{table}


\bigskip
\noindent {\bf Example 4.} We consider reaction-diffusion systems with stiff reaction terms in 2D and 3D spatial domains
\begin{align}
\begin{cases}
u_t = ka \triangle_{\mathbf{x}} u -bu +v, \\
v_t = ka \triangle_{\mathbf{x}} v -cv.
\end{cases}
\end{align}
The computational domains are $[0, 1]^2$ for the 2D case and $[0, 1]^3$ for the 3D case, with periodic boundary conditions. $k$, $a$, $b$ and $c$ are all constants. $k=\frac{1}{4d\pi^2}$ where $d$ denotes the spatial dimensionality. The initial condition is taken as
\begin{align}
\begin{cases}
u(\mathbf{x}, 0) = 2 \prod\limits_{i=1}^{d} \cos(2\pi x_i), \\
v(\mathbf{x}, 0) = \prod\limits_{i=1}^{d} \cos(2\pi x_i).
\end{cases}
\end{align}
The exact solution is
\begin{align}
\begin{cases}
u(\mathbf{x}, t) = (e^{-(b+a)t} +e^{-(c+a)t}) \prod\limits_{i=1}^{d} \cos(2\pi x_i), \\
v(\mathbf{x}, t) = (b-c)e^{-(c+a)t}\prod\limits_{i=1}^{d} \cos(2\pi x_i).
\end{cases}
\end{align}
We take parameters $c=a=1$ and $b=100$ such that the system involves stiff reaction terms. The system is solved by the Krylov IIF2 and Krylov IIF3 schemes with sparse grid IPDG spatial discretizations, with Krylov subspace dimension $M=25$ and time step size $\Delta t= h_N$. $L^2$ errors and the numerical orders of accuracy reported in Table \ref{2d3dsystem} show that the methods attain the desired accuracy. 
The numerical accuracy order of the $P^2$ case of the 3D computation is slightly less than $3$. Similar as other examples, large time step size (proportional to the spatial grid size) computations are obtained for the stiff reaction-diffusion system defined on
high spatial dimensions.

\begin{table} [htb]
\small
\begin{center}
\caption {Accuracy test for {\bf Example 4}. Krylov IIF2 and IIF3 schemes with sparse grid IPDG in space. $\Delta t = h_N$. N is the mesh level and DOF denotes the degrees of freedom of DG approximation.}
\begin{tabular}{ccccccccccccc}
  \cline{1-13}
     \multicolumn{1}{c}{\multirow{2}{*}{N}}  & \multicolumn{1}{c}{\multirow{2}{*}{DOF}}   & $P^1$  & && & &\multicolumn{1}{c}{\multirow{2}{*}{N}}  & \multicolumn{1}{c}{\multirow{2}{*}{DOF}} & $P^2$ \\
  \cline{3-7}
  \cline{10-13}
     & & $u$-error  & Order   & $v$-error & Order&  & &&   $u$-error & Order  & $v$-error & Order  \\
 \cline{1-13}
    \multicolumn{2}{c}{2D sparse}  & T=1.0\\
    5 & 448 & 1.44E-03   & --   & 1.42E-01 & --& & 5   &  1008  & 4.50E-05& --& 4.45E-03& -- \\
     6 & 1024 & 3.75E-04   & 1.94   & 3.72E-02 & 1.94& &6    &  2304  & 8.13E-06 & 2.47& 8.05E-04& 2.47 \\
     7 & 2304 & 1.06E-04   & 1.82   & 1.05E-02 & 1.82& &7    &  5184  & 1.74E-06 & 2.22& 1.72E-04& 2.22 \\
     8 & 5120 & 2.80E-05   & 1.92   & 2.77E-03 & 1.92& &8    &  11520  & 2.42E-07 & 2.85& 2.58E-05& 2.74 \\
    \cline{1-13}
       \multicolumn{2}{c}{3D sparse}  & T=0.2\\
    5 & 2176 & 1.18E-02   & --   & 1.17E-00&-- & &4    & 2808   & 2.01E-03 & --& 2.01E-01& -- \\
    6 & 5504 & 6.46E-03   & 0.87   & 6.40E-01 & 0.87& &5    &  7344  & 5.61E-04 & 1.84& 5.57E-02& 1.85 \\
    7 & 13568 & 1.86E-03   & 1.80  & 1.84E-01 & 1.80& &6    &  18576  & 1.20E-04 & 2.23& 1.19E-02&2.23 \\
    8 & 32768 & 6.46E-04   & 1.52   & 6.40E-02 & 1.52& &7    &  45792 & 2.29E-05 & 2.39& 2.26E-03& 2.39 \\
     9 & 77824 & 1.96E-04   & 1.72   & 1.94E-02 & 1.72& &8    &  110592  & 4.32E-06 & 2.40& 4.28E-04& 2.40 \\
    \cline{1-13}
   \end{tabular}\label{2d3dsystem}
\end{center}
\end{table}

\bigskip
\noindent {\bf Example 5.}  \textbf{Schnakenberg model.} The Schnakenberg system \cite{schnakenberg1979simple} has been used to model the spatial distribution of morphogens. It has the following form
\begin{align}
 \begin{cases}\label{SMeqn1}
 \vspace{0.1in}
 \displaystyle{\frac{\partial C_a}{\partial t}}=  D_1 \nabla^2 C_a+\kappa(a-C_a+C_a^2C_i),  \\
 \displaystyle{\frac{\partial C_i}{\partial t}}=  D_2 \nabla^2 C_i+\kappa(b-C_a^2C_i),
  \end{cases}
  \end{align}
where $C_a$ and $C_i$ represent the concentrations of activator and inhibitor, with $D_1$ and $D_2$ as the diffusion coefficients respectively. $\kappa$, $a$ and $b$ are rate constants of biochemical reactions. Following the setup in \cite{hundsdorfer2013numerical, zhu2009application}, we take the initial conditions as
\begin{align}\label{SMeqn2}
\displaystyle C_a(x,y,0)& = a + b + 10^{-3}e^{-100((x-\frac{1}{3})^2+(y-\frac{1}{2})^2)}, \\
\displaystyle C_i (x,y,0)& =  \frac{b}{(a+b)^2},
\end{align}
and the boundary conditions are periodic. The parameters are $\kappa = 100$, $a = 0.1305$, $b = 0.7695$, $D_1 = 0.05$ and $D_2 = 1$. The computational domain is $[0, 1] ^2$. We simulate this problem by the third order Krylov IIF scheme with $P^2$ sparse grid IPDG, with $N=8$, time step size $\Delta t = h_N$ and the Krylov subspace dimension $M=100$.  The numerical results for the concentration of the activator $C_a$ at different times are shown in Figure \ref{fig_pig}, from which we can observe that the proposed method is capable of producing the similar spot-like patterns as these in the literature \cite{zhu2009application, christlieb2015high1}. It is interesting to note that in order to observe the pattern formation with good resolution, Krylov subspace used in this example needs to have a larger dimension than those in previous examples.

\begin{figure}[!htb]
  \begin{minipage}[b]{0.3\textwidth}
    \centerline{
    \includegraphics[width=2.0in,angle=0,scale=1.0]{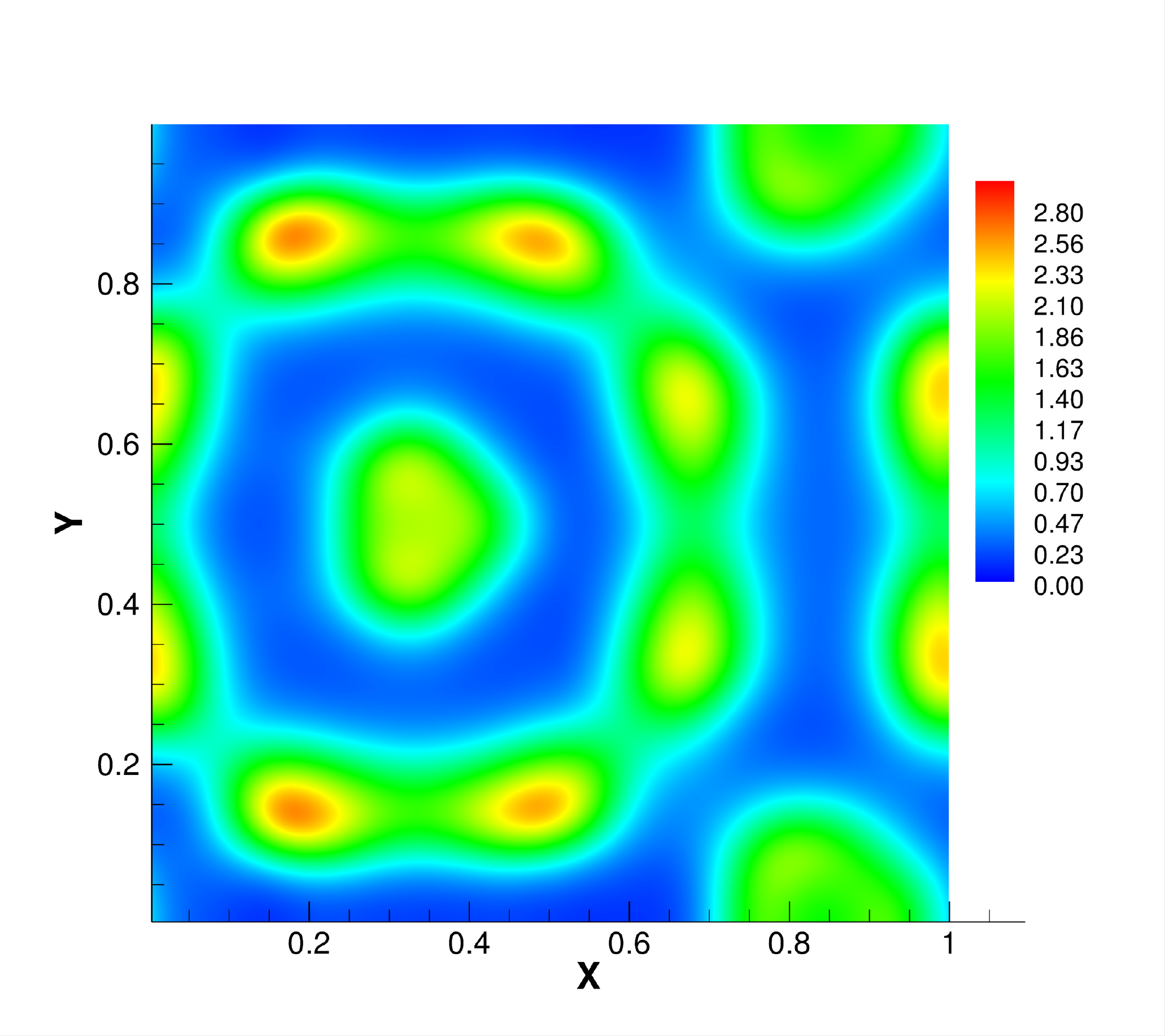}}
      \qquad \qquad t=0.5
    \medskip
  \end{minipage}%
  \hspace{0.04\linewidth}%
  \begin{minipage}[b]{0.3\textwidth}
    \centerline{
    \includegraphics[width=2.0in,angle=0,scale=1.0]{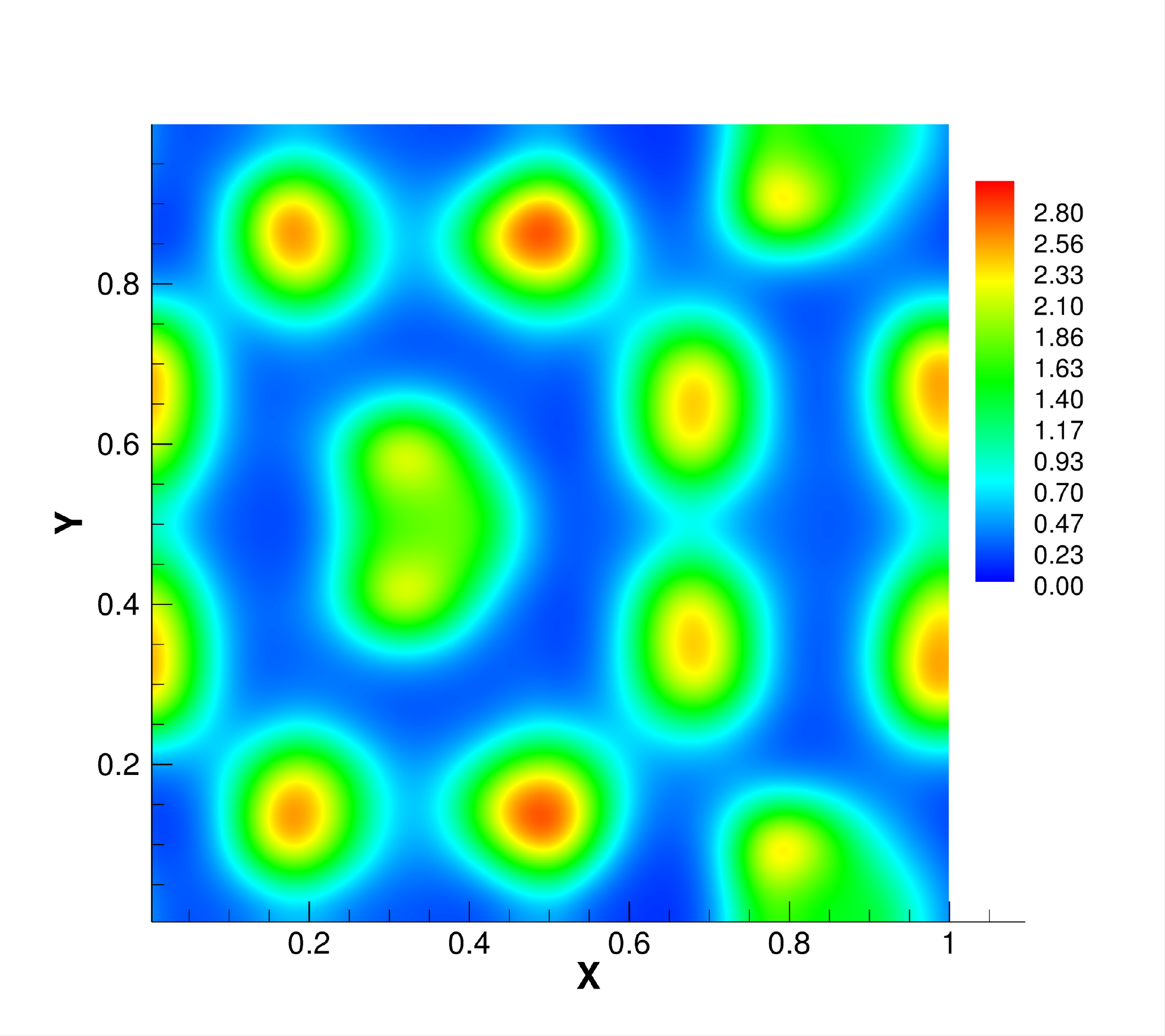}}
  \qquad \qquad t=0.6
  \medskip
  \end{minipage}
  \hspace{0.03\linewidth}%
  \begin{minipage}[b]{0.3\textwidth}
    \centerline{
    \includegraphics[width=2.0in,angle=0,scale=1.0]{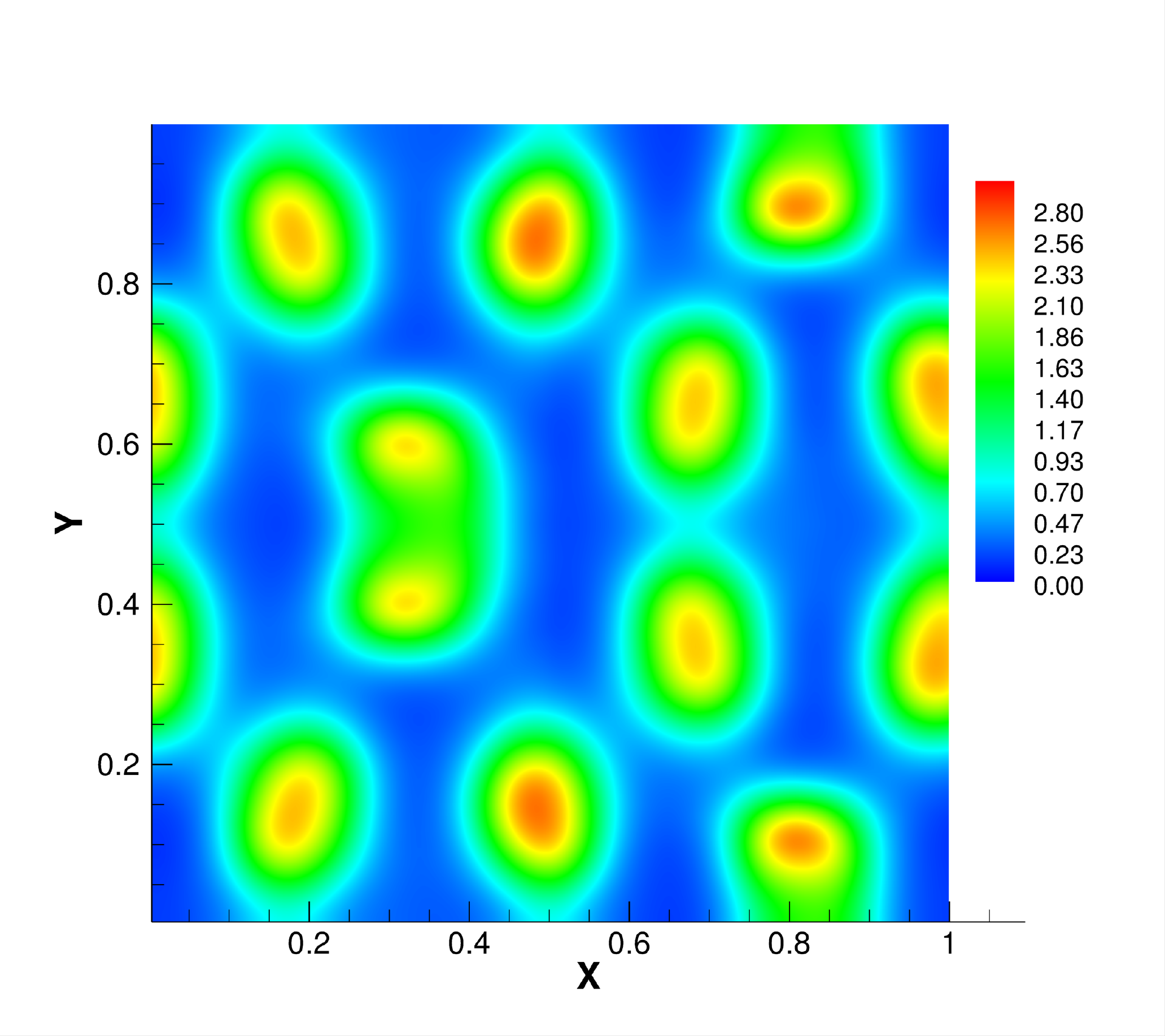}}
   \qquad \qquad t=0.7
  \medskip
  \end{minipage}
  \hspace{0.1\textwidth}%
  \begin{minipage}[b]{0.3\textwidth}
    \centerline{
    \includegraphics[width=2.0in,angle=0,scale=1.0]{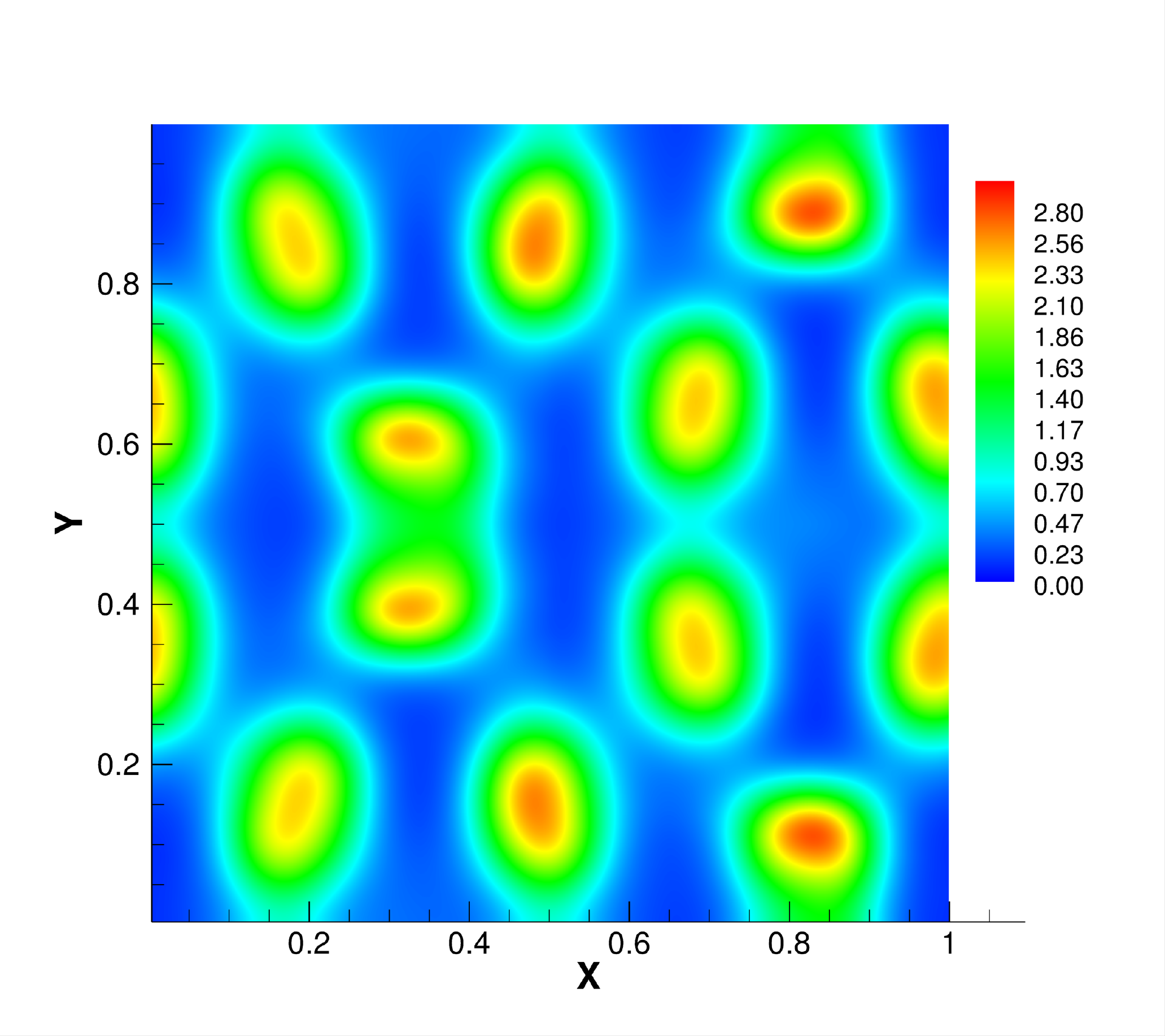}}
   \qquad \qquad t=0.8
   \medskip
  \end{minipage}
  \hspace{0.02\textwidth}
  \begin{minipage}[b]{0.3\textwidth}
    \centerline{
    \includegraphics[width=2.0in,angle=0,scale=1.0]{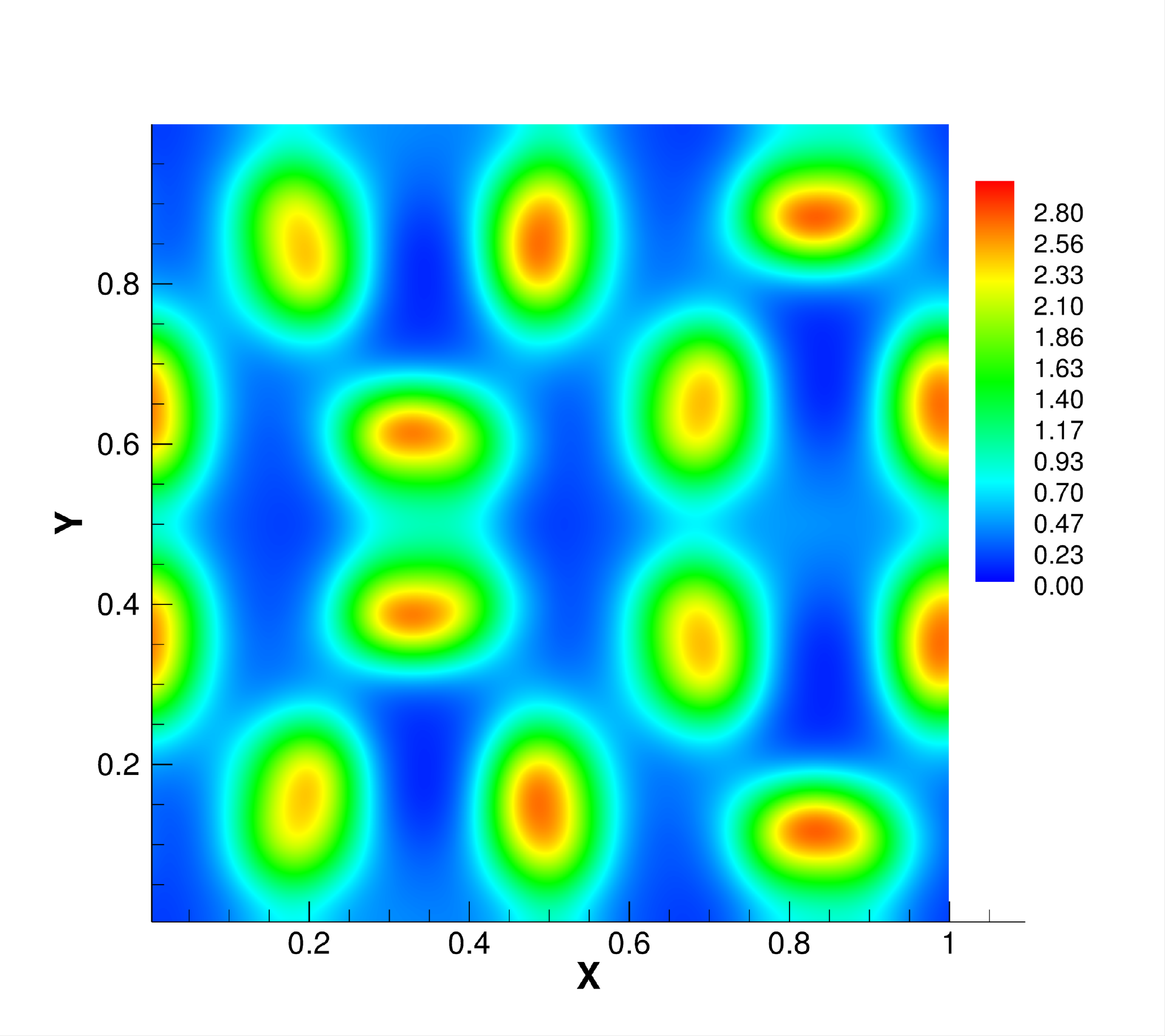}}
    \qquad \qquad t=1
    \medskip
  \end{minipage}
   \hspace{0.02\textwidth}
  \begin{minipage}[b]{0.3\textwidth}
    \centerline{
    \includegraphics[width=2.0in,angle=0,scale=1.0]{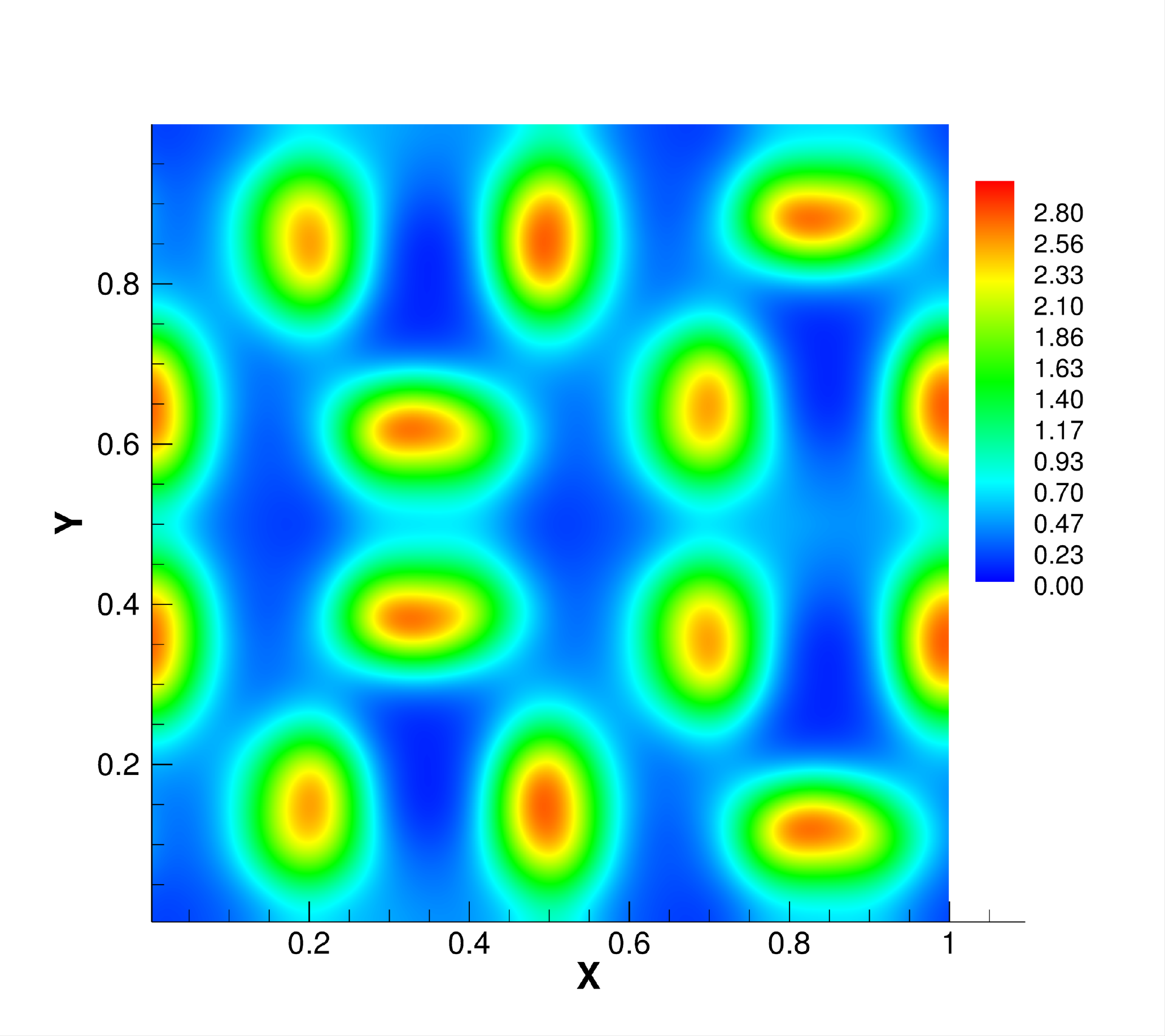}}
     \qquad \qquad t=1.5
    \medskip
  \end{minipage}
    \caption{Numerical simulations of the Schnakenberg reaction-diffusion model. Concentrations of the activator $C_a$ at different times. Third order Krylov IIF method with $P^2$ sparse grid IPDG. N=8, $\Delta t = h_N$, $M=100$. } \label{fig_pig}
\end{figure}

 \section{Discussions and conclusions}

Numerically solving multidimensional reaction-diffusion equations is computationally challenging due to
(1) stiffness of the system resulting from high order spatial differential operators and multiscale reaction
rates of different species; (2) large number of spatial grid points associated with high spatial dimensions.
In this paper, we design the Krylov IIF sparse grid DG methods to efficiently solve reaction-diffusion equations
defined on high spatial dimensions. The methods resolve the stiffness of DG spatial discretizations of reaction-diffusion equations via implicit exponential integration in the time direction. Expensive computational costs due to the high spatial dimensionality are taken care of by sparse grid techniques. Numerical experiments of both linear and nonlinear problems show that the methods can achieve both good accuracy and stable computations by using large time step size proportional to the spatial grid size. The stiffness of reaction-diffusion equations is resolved well, and significant
computational cost is saved by using sparse grid techniques.

There are two aspects which can be greatly improved and that are our future research directions. One is how to evaluate of the nonlinear integral $\int_{\Omega} f(u)v d\mathbf{x}$ efficiently, by taking into account the special property of basis function. In the current method, we directly reconstruct numerical approximation $u$ at the Gaussian points of finest level mesh and then calculate the integral by quadrature rules, which does not take advantage of sparsity and special property of the basis functions.  For polynomial type of nonlinearity as in the numerical examples of this work, the computation can be accelerated by computing the tensor products of the multiwavelet bases, but still there is room for improvement for the computation of general nonlinear term $f(u)$.
The other future work is how to preserve the ``local implicit property'' of the Krylov IIF schemes.
One of the nice properties of the IIF and Krylov IIF schemes is that the implicit terms do not involve the matrix exponential operations. Hence when the IIF or Krylov IIF schemes are applied to reaction-diffusion systems of PDEs with finite difference spatial
discretizations as that in \cite{NZZ,LZ1}, or DG spatial discretizations on regular grids \cite{chen2011krylov},
the size of the nonlinear system arising from the implicit treatment is independent of the number of spatial grid points; it only depends on the number of the original PDEs. Solving large size coupled nonlinear algebraic systems
at every time step is avoided and computations are very efficient. This ``local implicit property'' distinguishes IIF and Krylov IIF schemes from many other implicit schemes for stiff PDEs. In this paper, the ``local implicit property'' is not preserved  due to the special choice of DG basis functions on sparse grids.
Because our sparse grid DG approximation is a global representation, the size of the nonlinear algebraic system is the number of degrees of freedom of DG approximation. Although matrix exponentials are not involved in the implicit terms of the schemes, a large nonlinear algebraic system still needs to be solved by Newton iterations at every time step for problems with nonlinear reaction terms. It will be very interesting to figure out an approach to decouple the large nonlinear algebraic system into small size ones to achieve more efficient computations for the proposed sparse grid Krylov IIF DG schemes.



\end{document}